\documentclass[a4paper,12pt]{article}

\usepackage{amssymb}
\usepackage{epsfig}
\usepackage{amsfonts}
\usepackage{amsmath}
\usepackage{euscript}
\usepackage{amscd}
\usepackage{amsthm}
\usepackage{theoremref}
\usepackage{enumerate}
\usepackage{array}
\DeclareMathAlphabet{\mathpzc}{OT1}{pzc}{m}{it}

\newtheorem{thm}{Theorem}[section]
\newtheorem{lem}[thm]{Lemma}
\newtheorem{prop}[thm]{Proposition} 
\newtheorem{cor}[thm]{Corollary}
\newtheorem{rem}[thm]{Remark}

\newtheorem*{Defn}{Definition}

\newcommand{\Z}{\mathbb Z}

\newcommand{\N}{\mathbb \Z_{>0}}

\newcommand{\pZ}{\Z_{\geq 0}}

\newcommand{\td}{\operatorname{tr.deg}}

\newcommand{\ml}{\operatorname{ML}}
\newcommand{\gr}{\operatorname{gr}}

\newcommand{\p}{\mathpzc{p}}
\newcommand{\Ba}{B_{(a_1,a_2,\dots,a_n)}}
\newcommand{\A}{A^{\phi}}
\newcommand{\B}{B^{\phi}}
\newcommand{\D}{D^{\phi}}

\newcommand{\m}{\phi^{(m)}}
\newcommand{\Bh}{B^{\widehat{\phi}}}
\newcommand{\coloneq}{:=}

\title{On rigidity of Pham-Brieskorn surfaces}

\author{Neena Gupta* and Ananya Pal**\\
	{\small{\it Stat-Math  Unit, Indian Statistical Institute,}}\\
	{\small{\it 203 B.T. Road, Kolkata 700 108, India}}\\
	{\small{\it e-mail* : neenag@isical.ac.in, rnanina@gmail.com}}\\
	{\small{\it e-mail** : palananya1995@gmail.com}}
}
\begin{document}
	\date{}
	\maketitle
	
	\abstract
	It is well known that, over an algebraically closed field $k$ of characteristic zero, 
	for any three integers $a,b,c\geq 2$, any Pham-Brieskorn surface $B_{(a,b,c)}:= k[X,Y,Z]/(X^a + Y^b + Z^c)$ is rigid when at most one of $a,b,c$ is 2 and stably rigid when $\frac{1}{a} + \frac{1}{b} + \frac{1}{c}\leq 1$. In this paper we consider Pham-Brieskorn domains over an arbitrary field $k$ of characteristic $p\geq 0$ and give sufficient conditions on $(a,b,c)$ for which any Pham-Brieskorn domain $B_{(a,b,c)}$ is rigid.
	This gives an alternative approach to showing that there does not exist any non-trivial exponential map on $k[X,Y,Z,T]/(X^mY+T^{p^rq} + Z^{p^e})= k[x,y,z,t]$, for $m,q>1$, $p\nmid mq$ and $e>r\geq 1$, fixing $y$, a crucial result used in \cite{G1} to show that the Zariski Cancellation Problem (ZCP) does not hold for the affine $3$-space.
	
	We also provide a sufficient condition for $B_{(a,b,c)}$ to be stably rigid. 
	Along the way we prove that for integers $a,b,c\geq 2$ with $gcd(a,b,c) = 1$ and for $F(Y)\in k[Y]$, the ring $k[X,Y,Z]/(X^aY^b + Z^c+ F(Y))$ is a rigid domain.
	\smallskip
	
	\noindent
	{\small {{\bf Keywords.}} Pham-Brieskorn surface, Exponential map, rigid domain, stably rigid domain, Makar-Limanov invariant.
		
	}
	
	\section{\textbf{Introduction}}
	
	Throughout the paper $k$ will denote a field of arbitrary characteristic (unless mentioned otherwise specifically) and for any field $F$, $\overline{F}$ will denote its algebraic closure. All rings considered in this paper are commutative with unity and any ring homomorphism fixes unity. For a subring $A$ of $B$, we write $B = A^{[n]}$ to mean that $B$ is isomorphic to the polynomial algebra in $n$ variables over $A$. Capital letters like $X, Y, Z, T, U, V, X_1, \dots,X_n$ etc., will designate indeterminates over respective ground rings or fields. $ \Z, \pZ,\N$ stand for their usual meanings of all integers, all non-negative integers and all positive integers respectively.

	A $k$-algebra $A$ with no non-trivial exponential map (for definition of exponential map, see section \ref{exp}) is called {\it rigid}. Otherwise it is called a non-rigid ring.
	A $k$-algebra $A$ is said to be {\it stably rigid} if for every integer $N\geq 0$ and every exponential map $\phi$ on $A[X_1,X_2,.....,X_N] (= A^{[N]})$, $A\subseteq A[X_1,X_2,....,X_N]^{\phi}.$
	Note that stable rigidity implies rigidity.
	
	In \cite{G1}, the first author has shown that, over a field $k$ of positive characteristic $p$,  the Asanuma threefolds of the form 
	$$
	\dfrac{k[X,Y,Z,T]}{(X^mY + Z^{p^e} + T + T^{sp})}, 
	\text{~where~} m,e,s\in \N \text{~with~}  p^e\nmid sp,\, sp\nmid p^e \text{~and~} m>1,
	$$ 
	are counterexamples to the Zariski Cancellation Conjecture. 
	A crucial step in the proof was to show that on any ring of the form  
	$$
	B:=\dfrac{k[X,Y,Z,T]}{(X^mY + T^{sp^r} + Z^{p^e})},
	\text{~where~} m,s,r,e\in \N \text{~with~} p^e \nmid sp^r,\, sp^r\nmid p^e \text{~and~} m >1,
	$$ 
	there does not exist any non-trivial exponential map fixing $y$ (the image of $Y$ in $B$) and hence the $k(y)$-algebra
	$$
	B \otimes_{k[y]} k(y):=\frac{k(y)[X,Z,T]}{(y X^m + T^{sp^r} + Z^{p^e})}; \text{~where~} m,s,r,e\in\N \text{~with~} p^e \nmid sp^r,\, sp^r\nmid p^e\text{~and~} m > 1$$
	is rigid.
	
	In view of the importance of the rigidity of such surfaces, we undertake classification of  rings of the form $k[X,Y,Z]/(X^a + Y^b + Z^c)$ for $(a,b,c)\in \N^3$, in terms of rigidity.
	
	\medskip
	
	With the help of Eisenstein's criterion for irreducibility, one can easily conclude that, over a field $k$ of characteristic $p\geq 0$, for $n\geq 3$ and $a_1,a_2,\dots ,a_n \in \N$, the polynomial $X_1^{a_1} + X_2^{a_2}+ \dots+ X_n^{a_n} \in k[X_1,X_2,\dots ,X_n]\, (= k^{[n]}$) is irreducible if and only if $p\nmid gcd(a_1,a_2,\dots ,a_n)$. For $n\geq 3$ and $n$-tuple, $\underline{a}:=(a_1,a_2,\dots ,a_n)\in \N^n$, let 
	\begin{center}
		$B_{\underline{a}} = \Ba\coloneq\dfrac {k[X_1,X_2,\dots ,X_n]}{(X_1^{a_1} + X_2^{a_2}+ \dots+X_n^{a_n})}$.
	\end{center} 
	
	\noindent
	When $k$ is a field of characteristic $0$, for $n\geq 3$ and an arbitrary $n$-tuple ($a_1,a_2,\dots ,a_n)\in \N^n$, the integral domains of the form $\Ba$ are known as Pham-Brieskorn rings.
	We shall call the rings of the form $\Ba$ Pham-Brieskorn rings (or domains) even over a field $k$ of arbitrary characteristic $p$.
	These rings and the corresponding varieties have been studied extensively from several angles in the characteristic zero setup (\cite[Section 9.2]{Fr}, \cite{FL}, \cite{CD}) and, as the results in \cite{G1} suggest, it now needs to be settled which Pham-Brieskorn domains are rigid or stably rigid in arbitrary characteristic.
	\medskip
	
	\noindent
	For a field $k$ of characteristic $p\geq 0$ and for any $n\geq 3$, we consider the sets
	\begin{flushleft}
		$F_n\coloneq\{(a_1,a_2,\dots ,a_n)\in \N^n \mid p\nmid gcd(a_1,a_2,\dots, a_n)\}$,\\
		$T_n:= \{(a_1,a_2,\dots ,a_n)\in \N^n\mid a_i = 1 \text{~for some~} i, \text{~or~} \exists\,  i,j\in \{1,2,\dots, n\}, \, i\neq j,\,  a_i = a_j =2\}$ \\
		and\\
		$R_n\coloneq\{(a_1,a_2,\dots ,a_n)\in \N^n\mid a_i = 1 \text{~for some~} i, \text{~or~}\exists\, i,j\in \{1,2,\dots, n\},\, i\neq j,\, a_i=p^r, a_j= sp^e, r,s,e\in\N ,\,r\leq e\}$.
	\end{flushleft} 
	By \cite[section 9.2]{Fr} and \cite[Theorem 7.1]{FL}, it is known that for an algebraically closed field of characteristic $0$, $B_{(a,b,c)}$ is rigid if and only if $(a,b,c)\not\in T_3$. In this paper we will establish the following results on Pham-Brieskorn rings.\\
	
	\noindent
	\textbf{Theorem A.} (Result on stable rigidity, Section \ref{SR})
	
	\noindent
	Over any field $k$ of arbitrary characteristic $p\geq 0$, the Pham-Brieskorn domain $B_{(a,s_2p^r,s_3p^e)}$ is stably rigid when $a,s_2,s_3\in\N$ with $\frac{1}{a} + \frac{1}{s_2} + \frac{1}{s_3}\leq 1,\, r,e\in\pZ$ and $p\nmid as_2s_3$ (\thref{Stably rigid 1}).
	
	\hspace{0.5cm}
	
	\noindent
	\textbf{Theorem B.} (Results on rigidity, Section \ref{ThmB})\label{B}
	
	\noindent
	For any field $k$ of characteristic $p\geq 0$ and for an arbitrary integer $n\geq 3$,	
	\begin{enumerate}[\rm(i)]
		
		\item $\Ba$ is non-rigid if $(a_1, a_2,\dots ,a_n)\in R_n$ (subsection \ref{(I)}).
		
		\item $\Ba$ is non-rigid when $(a_1,a_2,\dots,a_n)\in T_n$ and $k$ contains a square root of $-1$ (subsection \ref{(I)}).
		
		\item 
		$ \Ba$ is non-rigid when $(a_1,a_2,\dots,a_n)\in \N^n\setminus F_n$ (subsection \ref{(I)}).
		
		\item 
		$ B_{(a,b,c)}$ is rigid when $(a,b,c)\in F_3\setminus (R_3\cup T_3\cup S_3),\text{where~} S_3 :=\{(2,2m,2p^e)\mid m, e \in \N, \, m>1 \text{~and~} p\nmid 2m \}$ (\thref{main}).
	\end{enumerate}
	In section \ref{AS}, we prove the rigidity of surfaces of type $k[X,Y,Z]/(X^a + Y^bZ^c +F(Y))$ for $a,b,c\geq 2$, $gcd(a,b,c)=1$ and $F(Y)\in k[Y]$ (\thref{Cor}). In section \ref{app}, we present a few applications of our results.
	
	\section{\textbf{Preliminaries}}
	
	For a ring $A$, $A^*$ denotes its group of units. If $A\subseteq B$ are domains then $\td_A(B)$ denotes the transcendence degree of Frac($B$) over Frac($A$), where Frac($A$) is the fraction field of $A$. 
	
	\subsection{Exponential map}\label{exp}
	
	We first recall the concept of an exponential map on a $k$-algebra.
	
	\smallskip
	
	\noindent
	\begin{Defn}
		Let $A$ be a $k$-algebra and let $\phi_U: A \rightarrow A[U]$ be a $k$-algebra homomorphism. We say that $\phi = \phi_U$ is an exponential map on $A$, if $\phi$ satisfies the following two properties:
		\begin{enumerate}[\rm(i)]
			\item $\varepsilon_0\phi_U$ is identity on $A$, where $\varepsilon_0: A[U]\rightarrow A$ is the evaluation map at $U = 0$.
			
			\item $\phi_V\phi_U = \phi_{V+U}$; where $\phi_V: A\rightarrow A[V]$ is extended to a homomorphism $\phi_V:A[U]\rightarrow A[U,V]$ by defining $\phi_V(U) = U$.	
		\end{enumerate}		
	\end{Defn}
	
	The above definition can also be rewritten in terms of the following five properties.
	
	\medskip
	
	\noindent 
	For each $a\in A$, we write $\phi(a) = \sum_{i=0}^{\infty}\phi^{(i)}(a)U^i$ in $A[U]$.
	\begin{enumerate}[\rm(i)]
		\item The sequence of maps $\{ \phi^{(i)}\}_{i=0}^{\infty}$ is a sequence of linear maps on $A$.
		
		\item For each $a\in A$, the sequence $\{ \phi^{(i)}(a)\}_{i=0}^{\infty}$ has only finitely many non-zero terms.
		
		\item $\phi^{(0)}$ is the identity map on $A$.
		
		\item (Leibniz Rule) For all integers $n\geq 0$ and for $a,b\in A$, $\phi^{(n)}(ab) =\sum\limits_{i+j=n}\phi^{(i)}(a)\phi^{(j)}(b)$.
		
		\item For all non negative integers $i$ and $j$, $\phi^{(i)}\phi^{(j)}= \binom{i+j}{i}\phi^{(i+j)}$.
	\end{enumerate}
	
	\noindent	
	Given an exponential map $\phi =\phi_U$ on $A$, we can define the $\phi$-degree of an element $a \in A\setminus \{ 0 \}$, by $\deg_{\phi}(a)= \deg_U(\phi(a))$ and $\deg_{\phi}(0) = -\infty$. The ring of $\phi$-invariants of $\phi_U$ is a subring of $A$ given by 
	\begin{center}
		$\A = \left\{ a\in A\,|\, \phi(a) = a\right\}=
		\{a \in A\mid \deg_{\phi}(a) \leq 0\} = \{ a\in A\mid \deg_U(\phi(a)) = 0\}\cup \{0\}$.
	\end{center}
	Moreover, if $A$ is an integral domain then the function $\deg_{\phi}$ is a degree function (for definition please refer to subsection \ref{II}) on $A$.
	
	An exponential map $\phi$ is said to be non-trivial if $\A \ne A$. For an affine $k$-algebra $A$, let EXP(A) denote the set of all exponential maps on $A$.
	The 
	{\it Makar-Limanov invariant} of $A$ is a subring of $A$ defined by
	\begin{center}
		$\ml(A) = \bigcap\limits_{\phi \in \text{EXP}(A)}\A$.
	\end{center} 
	Let $\ml_0(B) = B$ and for $n\in \N$, $\ml_n(B):= \ml(\ml_{n-1}B).$ Then the {\it rigid core} of $B$, denoted by $\mathcal{R}(B)$, is defined as 
	$$\mathcal{R}(B) = \bigcap\limits_{n\geq 1}\ml_n(B).$$
	
	Recall a $k$-algebra $A$ is said to be rigid if it has no non-trivial exponential map and stably rigid if $A\subseteq \ml(A^{[N]})$, for every integer $N\geq 0$.
	
	We list below some useful properties of exponential maps (cf. \cite{Cr} and \cite{G1}).
	
	\begin{lem}\thlabel{lemma : properties}
		Let $A$ be an affine domain over $k$, and let $\phi$ be a non-trivial exponential map on $A$. Then the following holds: 	\begin{enumerate}[\rm(i)]			
			\item $\A$ is factorially closed in $A$ i.e., for any $a,b\in A\setminus\{0\}$, if $ab\in \A$ then $a,b\in\A$. Hence, $\A$ is also algebraically closed in $A$.
			
			\item $\td_k(\A)$ = $\td_k(A)-1$.
			
			\item $\deg_{\phi}(\phi^{(i)}(a)) \leq \deg_{\phi}(a)-i$, for all $a\in A$, $i\in \N$. If $a\neq 0$ then $\phi^{(\deg_{\phi}a)}(a) \in \A$.
			
			\item Suppose $x\in A$ have the minimal positive $\phi$-degree n and $c:= \phi^{(n)}(x)$. Then  $c\in\A$ and $A[c^{-1}] = \A[c^{-1}][x] = \A[c^{-1}]^{[1]}$.
			
			\item If $\td_k(A) = 1$, then $A= \widetilde{k}^{[1]}$, where $\widetilde{k}$ is the algebraic closure of $k$ in $A$ and $\A = \widetilde{k}$.
			
			\item Let $S$ be a multiplicative closed subset of $\A\setminus \{0\}$. Then $\phi$ extends to a non-trivial exponential map $S^{-1}\phi$ on $S^{-1}A$ defined by $S^{-1}\phi(a/s) = \phi(a)/s$, for all $a\in A$ and $s\in S$. Moreover, the ring of invariants of $S^{-1}\phi$ is $S^{-1}(\A)$.	
		\end{enumerate}	
	\end{lem}
	
	\begin{rem}\thlabel{Remark 2}
		{\em\begin{enumerate}[\rm(i)]
				\item 	Whenever for a $k$-domain $A$, $A\otimes_k\overline{k}$ is a rigid ring then $A$ is also rigid, since an exponential map $\phi$ on $A$ can be extended to an exponential map ${\phi}\otimes id$ on $A\otimes_k\overline{k}$.
				
				\item Let $A$ be a $k$-algebra such that  $A\otimes_k\overline{k}$ is an integral domain and $\td_kA = 1$. Then $k$ is algebraically closed in $A$ and hence $A$ admits a non-trivial exponential map if and only if $A = k^{[1]}$.
			\end{enumerate}
		}
	\end{rem}
	
	\subsection{Homogenization of an exponential map}\label{II}
	
	We first define a homogeneous exponential map over a $\Z$-graded $k$-algebra.
	
	\smallskip
	
	\noindent
	\begin{Defn}
		Let $\phi$ be an exponential map on a  $\Z$-graded $k$-algebra $A = \bigoplus_{n\in \Z} A_n$ and for any $x\in A$, let $\phi(x) := \sum_{i=0}^{\infty}\phi^{(i)}(x)T^i$ in $A[T]$. Then $\phi$ is said to be a homogeneous exponential map if there exists some $d\in \mathbb{Q}$ such that for every homogeneous element $a$ of degree $r$, $\phi^{(i)}(a)$ is a homogeneous element of degree $r + id$ i.e. $\phi^{(i)}(A_r) \subseteq A_{r+id}$, for all $i\in \pZ$ and $r\in \Z$.	
	\end{Defn}
	
	For this subsection $A$ denotes an affine $k$-domain. We define below a proper and admissible $\Z$-filtration on $A$.
	
	\smallskip
	
	\noindent
	\begin{Defn}
		A collection $\{ A_n\mid n\in \Z \}$ of $k$-linear subspaces of $A$ is said to be a proper $\Z$-filtration if the following properties hold:
		\begin{enumerate}[\rm(i)]
			\item $A_n\subseteq A_{n+1}$ for every $n\in \Z$.
			
			\item $A = \cup_{n}A_n$.
			
			\item $\cap_n A_n = \{ 0 \}$.
			
			\item $(A_n \setminus A_{n-1})(A_m \setminus A_{m-1})\subseteq A_{m+n}\setminus A _{m+n -1}$ for all $m,n\in \Z$.
		\end{enumerate}	
	\end{Defn}
	
	Next we define a related concept of degree function on $A$.
	
	\smallskip
	
	\noindent
	\begin{Defn}
		A degree function on $A$ is a function $\deg : A\rightarrow\Z\cup \{-\infty\}$ satisfying:
		\begin{enumerate}[\rm(i)]
			\item $\deg(a)=-\infty$ if and only if $a=0$.
			
			\item $\deg(ab) = \deg(a)+ \deg(b)$. 
			
			\item $\deg(a+b)\leq max\{\deg(a), \deg(b)\}$ for all $a,b\in A$.
		\end{enumerate}
	\end{Defn}
	
	\begin{rem}\thlabel{Remark 3}
		{\em
			Any proper $\Z$-filtration $\{ A_n\}_{n\in\Z}$ on $A$ defines a degree function $\deg$ on $A$, where $\deg(0)=-\infty$ and $\deg(a)=\min\{d\in \Z\mid a\in A_d\}$ for $a\in A\setminus\{0\}$.
			Conversely any degree function, $\deg$ on $A$ determines a proper filtration on $A$ given by $A_n=\{ a\in A\mid \deg(a)\leq n\}$, for all $n\in \Z.$
		}
	\end{rem} 
	
	\noindent
	Any proper $\Z$-filtration on $A$ determines an  associated $\Z$-graded integral domain
	\begin{center}
		$\gr(A):= \bigoplus\limits_{i\in \Z} \frac{A_i}{A_{i-1}}$.
	\end{center}
	There exists a natural map $\rho: A \rightarrow \gr(A)$ defined by $\rho(a) = a + A_{n-1}$ if $a\in A_n\setminus A_{n-1}$.
	\begin{Defn}
		A proper $\Z$-filtration $\{A_n\}_{n\in \Z}$ on $A$ is said to be admissible if there exists a finite generating set $\Gamma$ of $A$ such that, for any $n\in\Z$ and $a\in A_n$, $a$ can be written as a finite sum of monomials in elements of $\Gamma$ and each of them is a monomial in $A_n$ also. 
	\end{Defn}
	\begin{rem}\thlabel{Remark 4}
		{\em 
			Suppose that $A$ has a proper $\Z$-filtration and a finite generating set $\Gamma$ which makes the filtration admissible. Then $\gr(A)$ is generated by $\rho(\Gamma)$ \cite[Remark 2.2]{G1}. 
		}
	\end{rem}
	
	\begin{rem}\thlabel{Remark 1}
		{\em 
			Let $A$ be a $\Z$-graded algebra say, $A = \bigoplus_{i\in\Z} C_i$. Then there exists a proper $\Z$-filtration $\{ A_n\}_{n\in\Z}$ on $A$ defined by $A_n\:= \bigoplus_{i\leq n}C_i$ and $\gr(A) = \bigoplus_{n\in \Z}A_n/A_{n-1}\cong \bigoplus_{n\in\Z}C_n = A$. Moreover, for any element $a\in A$, $\rho(a)$  is the highest degree homogeneous summand of $a$ in $A$. 
			Thus the above filtration on $A$ is admissible.
		}
	\end{rem}
	
	Next we record a theorem on homogenization of exponential maps by H. Derksen, O. Hadas and L. Makar-Limanov (\cite[Theorem 2.6]{Cr}).
	
	\begin{thm}\thlabel{Theorem: homogeneization}
		Let $A$ be an affine domain over a field $k$ with an admissible $\mathbb{Z}$-filtration. Suppose $\phi$ is a non-trivial exponential map on $A$. Then $\phi$ induces a non-trivial homogeneous exponential map $\widehat{\phi}$ on the associated graded domain $\gr(A)$ such that $\rho(A^{\phi}) \subseteq (\gr(A))^{\widehat{\phi}}$, where $ \rho$ is the natural map $A \rightarrow \gr(A)$.	
	\end{thm}	
	
	\subsection{Some known results}
	
	We now recall some well known results which will be used throughout the text.	
	First we state a formulation of the Epimorphism Theorem due to Abhyankar and Moh \cite[Theorem 1.1]{AM}.
	
	\begin{thm}\thlabel{EpiTheorem}
		Let $\varphi: k[X,Y]\rightarrow k[T]$ be a $k$-algebra epimorphism. Suppose $\deg_T(\varphi(X)) = n\geq 1$ and $\deg_T(\varphi(Y)) = m\geq 1$ with ch$(k)\nmid gcd(m,n)$. Then either $m\mid n$ or $n\mid m$.
	\end{thm}
	
	The next two lemmas are proved in 
	\cite[Lemma 3.3 and 3.4]{CM} respectively.
	
	\begin{lem}{\thlabel{Mini Mason 1}}
		Let $\phi$ be an exponential map on a $k$-domain $A$ with characteristic $p\geq 0$.
		Let $f,g \in A$ be such that there exist $a,b\in \N\setminus \{1\}$ and $c_1,c_2\in \A\setminus \{0\}$ for which
		$c_1f^a + c_2g^b \in \A\setminus \{0\}$ and neither $a$ nor $b$ is a power of $p$. Then $f,g\in \A$.
	\end{lem}
	
	\begin{lem}{\thlabel{Mini Mason 2}}
		Let $\phi$ be an exponential map on a $k$-domain $A$ with characteristic $p> 0$. 
		Let $f,g\in A$ be such that $f$ is prime in $A$ and there exist $a,l\in \N$, $a > 1$, $p\nmid a$ and $c_1,c_2\in \A\setminus \{0\}$
		for which $c_1f^a + c_2g^{p^l} \in \A\setminus \{0\}$. Then $f,g\in \A$.
	\end{lem}
	
	The next lemma from  \cite[Lemma 3.3]{G1} provides an important criterion for the existence of non-trivial exponential map on certain affine domains.
	
	\begin{lem}\thlabel{betta-lemma}
		Let $A$ be an affine domain over an infinite field $k$. Let $f\in A$ be such that $f-\lambda$ is a prime element of $A$ for infinitely many $\lambda$ in $k$. Let $\phi: A\rightarrow A[U]$ be a non-trivial exponential map on $A$ such that $f\in \A$. Then there exists a $\beta\in k$ such that $f-\beta$ is a prime element of $A$ and $\phi$ induces a non-trivial exponential map on $A/(f-\beta)A$.
	\end{lem}
	
	\begin{rem}\thlabel{Remark 6}
		{\em (cf. \cite[section 1]{FL})
			For an integral domain $A$, two non-zero elements $f,g$ of $A$ are said to be relatively prime in $A$ if $fA\cap gA = fgA$. The following hold when $f,g$ are relatively prime in $A$:
			\begin{enumerate}[\rm(i)]
				\item If $fh_1 = gh_2$ for some $h_1,h_2\in A$ then $g$ divides $h_1$ and $f$ divides $h_2$.
				\item $f^m$ and $g^n$ are relatively prime for all $m,n\in\N$.
				\item $f,g$ are relatively prime in $A^{[n]}$, for all $n\in \N$.
				\item $f,g$ are relatively prime in $A_1$, for any factorially closed subring $A_1$ of $A$ containing $f$ and $g$. 
			\end{enumerate}
		}
	\end{rem}
	
	\section{\textbf{Theorem A: Stable Rigidity of the Pham-Brieskorn domain} }\label{SR} 
	
	The following result on stable rigidity has been proved in \cite{KZ}, \cite[Theorem 6.1(a)]{FL} and \cite[Theorem 9.7]{Fr} over a field of characteristic zero. We have modified the arguments in \cite[Theorem 9.7]{Fr} to give a characteristic-free proof.
	
	\begin{prop}\thlabel{Theorem 1}
		Let $A$ be a $k$-domain of characteristic $p\geq 0$. Suppose $x,y,z\in A\setminus\{0\}$ are pairwise relatively prime in $A$ satisfying $x^a + y^b + z^c = 0$, for some positive integers $a,b,c$ with $p\nmid abc$ and $\frac{1}{a} + \frac{1}{b} + \frac{1}{c} \leq 1$. Then $k[x,y,z]\subseteq \ml(A)$. Moreover, $k[x,y,z]\subseteq \mathcal{R}(A)$.
	\end{prop}
	\begin{proof}
		Let $\phi$ be  a non-trivial exponential map on $A$. We want to show that $k[x,y,z]\subseteq \A$. 
		Suppose at least one of $x,y,z$ is in $\A$, say $z\in \A$. Then $x^a + y^b\in \A\setminus \{0\}$. Since $\frac{1}{a} + \frac{1}{b} + \frac{1}{c} \leq 1$, and hence $a,b,c\geq 2$,  by \thref{Mini Mason 1}, $x,y \in \A$ i.e. $k[x,y,z]\in \A$. Thus it suffices to show that at least one of $x,y,z$ is in $\A$. 
		
		Suppose, if possible, $x,y,z\notin \A$. Then there would exist a positive integer $m$ defined as follows:
		\begin{center}
			$m:=$  min$\{ i\in \N \mid $  $\phi^{(i)}(x)\neq 0$ or  $\phi^{(i)}(y)\neq 0$ or $\phi^{(i)}(z)\neq 0\}$.
		\end{center}
		Now $\phi^{(i)}(x)= \phi^{(i)}(y)= \phi^{(i)}(z)= 0$ for $1\leq i < m $ and $\phi^{(m)}(x^a + y^b + z^c) = 0$, so by Leibniz rule we would have, 
		\begin{equation}\label{1}
			ax^{a-1}\phi^{(m)}(x) + by^{b-1}\phi^{(m)}(y) + cz^{c-1}\phi^{(m)}(z) = 0.
		\end{equation}
		Hence, by the definition of $m$, at least any two of $\phi^{(m)}(x),\phi^{(m)}(y),\phi^{(m)}(z)$ would be non-zero, say $\phi^{(m)}(x) \neq 0$ and $\phi^{(m)}(y) \neq 0$. 
		
		Let $M$ be the 3$\times$3 matrix with elements from A defined by
		$$M = 
		\begin{bmatrix}
			x & y & z \\
			a\phi^{(m)}(x) & 	b\phi^{(m)}(y) & 	c\phi^{(m)}(z)\\
			0 & 0 & 1\\
		\end{bmatrix}.$$
		Then, by (\ref{1}) and the given condition $x^a + y^b + z^c = 0$, we would have
		$$M
		\begin{bmatrix}
			x^{a-1}\\y^{b-1}\\z^{c-1}\\
		\end{bmatrix} 
		= z^{c-1}
		\begin{bmatrix}
			0\\0\\1\\
		\end{bmatrix}.$$
		
		\noindent 
		Since $x,y$ are relatively prime in $A$ and $x\nmid\phi^{(m)}(x)$ (by \thref{lemma : properties}(iii)), it would follow that $bx\phi^{(m)}(y) - ay\phi^{(m)}(x) \neq 0$ and  hence det($M)\neq 0$.
		Let Adj($M) = (c_{ij})$.
		Then
		$$\text{det}(M) 
		\begin{bmatrix}
			x^{a-1}\\y^{b-1}\\z^{c-1}\\
		\end{bmatrix} = z^{c-1}\text{Adj}(M)
		\begin{bmatrix}
			0\\0\\1\\
		\end{bmatrix}= z^{c-1}
		\begin{bmatrix}
			c_{13}\\c_{23}\\c_{33}\\
		\end{bmatrix} = z^{c-1}
		\begin{bmatrix}
			cy\m(z) - bz\m(y)\\
			az\m(x) - cx\m(z)\\
			bx\m(y)-ay\m(x)\\
		\end{bmatrix},$$
		that is,
		$$\begin{array}{lll}
			\text{det}(M) x^{a-1} & = &  z^{c-1}(cy\m(z) - bz\m(y)) \text{~~and~~}\\
			\text{det}(M)y^{b-1} & = & z^{c-1}(az\m(x) - cx\m(z)). 
		\end{array}$$	
		Next, since $x^{a-1}, y^{b-1}$ and $z^{c-1}$ are pairwise relatively prime (cf. \thref{Remark 6} (ii)), it would follow
		$$\begin{array}{llll}
			x^{a-1} & \text{~~divides~~} & (cy\m(z) - bz\m(y)),\\
			y^{b-1} & \text{~~divides~~} & (az\m(x) - cx\m(z))\text{~~and~~}\\
			z^{c-1} & \text{~~divides~~} & (bx\m(y) - ay\m(x)) \text{~(= det}(M)). 
		\end{array}$$
		Therefore: 
		\begin{equation}\label{2}
			(a-1)\deg_{\phi}(x)\leq \deg_{\phi}(y) + \deg_{\phi}(z) - m, 
		\end{equation}
		\begin{equation}\label{3}
			\hspace{0.8cm}(b-1)\deg_{\phi}(y)\leq \deg_{\phi}(z) + \deg_{\phi}(x) - m\text{~and~}
		\end{equation}
		\begin{equation}\label{4}
			(c-1)\deg_{\phi}(z)\leq \deg_{\phi}(x) + \deg_{\phi}(y) - m.	
		\end{equation}	
		Let $j = \deg_{\phi}(x) + \deg_{\phi}(y) + \deg_{\phi}(z)$. Then from equations (\ref{2}), (\ref{3}) and (\ref{4}) we get 
		\begin{center}
			$j \leq (j - m)(\frac{1}{a} + \frac{1}{b} + \frac{1}{c})\leq (j -1)(\frac{1}{a} + \frac{1}{b} + \frac{1}{c})\leq (j -1)$,
		\end{center}
		a contradiction. 
		
		Therefore, at least one of $x,y,z$ and hence all of them are in $\A$; i.e. $k[x,y,z]\subseteq \A$. 
		
		Since $\phi$ is chosen arbitrarily, $k[x,y,z]\subseteq \ml(A)$.
		
		As $\ml(A)$ is a factorially closed domain (cf. \thref{lemma : properties}(i)), it follows that $x,y,z$ are also pairwise relatively  prime in $\ml(A)$ (cf. \thref{Remark 6}(iv)). Now a simple induction  shows that $k[x,y,z]\subseteq \ml_n(A)$, for all $n\in\N$. Therefore $k[x,y,z]\subseteq\mathcal{R}(A)$.	
	\end{proof}
	
	Now Theorem A follows from the above proposition. 
	
	\begin{thm}\thlabel{Stably rigid 1}
		Let $k$ be a field of characteristic $p\geq 0$.
		For any $a,s_2,s_3\in\N$, $r,e\in \pZ$ and $p\nmid as_2s_3$, the domain
		$$
		B_{(a,s_2p^r,s_3p^e)}:=\dfrac{k[X,Y,Z]}{(X^a + Y^{s_2p^r} + Z^{s_3p^e})}
		$$  
		is stably rigid when $\frac{1}{a} + \frac{1}{s_2} + \frac{1}{s_3}\leq 1$.
	\end{thm}
	\begin{proof}
		Set $B:=B_{(a,s_2p^r,s_3p^e)}$ and let $A = B^{[m]}$, for arbitrary $m\in\pZ$.  	
		Let $x,y,z$ be the images of $X,Y$ and $Z$ in $B$ respectively. Now $x,y^{p^r},z^{p^e}$ are all pairwise relatively prime in $A$ (cf. \thref{Remark 6} (ii) and (iii)).
		Further $x^a + (y^{p^r})^{s_2} + (z^{p^e})^{s_3} = 0$ in $A$ with $(\frac{1}{a}+\frac{1}{s_2} + \frac{1}{s_3})\leq 1$ and $p\nmid as_2s_3$.
		Therefore, by \thref{Theorem 1}, $x,y^{p^r},z^{p^e}\in \ml(A)$. Hence $B =k[x,y,z]\subset \ml(A)$, as $\ml(A)$ is a factorially closed domain (cf. \thref{lemma : properties}(i)). Thus $B$ is stably rigid.	
	\end{proof}

	\section{\textbf{Some Auxiliary Results on Rigidity}}\label{AS}
	
	In this section we are going to prove the rigidity of domains of the form $$\dfrac{k[X,Y,Z]}{(X^a + Y^bZ^c+ F(Y))} \text{~~where~} a,b,c\geq 2,\,  gcd(a,b,c)=1 \text{~and~} F(Y)\in k[Y].$$
	These results will be used in sections \ref{ThmB} and \ref{app}. 
	We first prove a technical lemma in this connection. 
	
	\begin{lem}
		{ Let $R$ be a $\pZ$-graded, finitely generated $k$-domain, $r\in \N$ and let $h\in R$ be such that 
			$$
			C:= \dfrac{R[X]}{(X^r + h)} \text{~ and ~} 
			\widehat{C}:= \dfrac{R[X]}{(X^r + \widehat{h})}
			$$ 
			are integral domains, where $\widehat{h}$ is the highest degree homogeneous summand of $h$.	 Then we can define an admissible $\Z$-filtration on $C$ such that $\gr(C)\cong \widehat{C}$. 
		}
	\end{lem}
	\begin{proof}
		Let $d$ denote the degree function induced on $R$ by the given $\pZ$-grading (cf. \thref{Remark 3}). 
		We can assume that $r$ divides $d(h)$, if necessary we can change the grading by multiplying the degree function suitably and set $a:=d(h)/r$. We regard $R$ as a subring of $C$. Let $x$ denote the image of $X$ in $C$. By using the relation $x^r =- h$, every $g$ in $C$ has a unique representation as
		\begin{equation}\label{r}
			g = \sum\limits_{j=0}^{r-1}g_jx^j\text{
				for some~} g_j\text{~in~} R.
		\end{equation} 
		Let $\deg: C\setminus \{0\}\rightarrow \pZ$ be a function defined by, 
		$$
		\deg(g)= \max_{0\leq j< r}\{d(g_j) + ja\}.
		$$
		It is clear that $\deg$ is a semi-degree function. We now show that for non-zero $f,g\in C$ with $\deg(f)= l$ and $\deg(g) = m$ we have $\deg{(fg)}= l+m$.
		This would imply that the above defined $\deg$ is a degree function on $C$. 
		
		Let $f,g\in C\setminus\{0\}$ with $\deg(f)= l$ and $\deg(g) = m$. Let $f_0$ and $g_0$ respectively denote the highest degree homogeneous summands of $f$ and $g$ in $C$. Let 
		$$f_0=\sum\limits_{j=0}^{r-1}f_{0j}x^j\text{~and~}g_0=\sum\limits_{i=0}^{r-1}g_{0i}x^i$$
		be the unique representation of $f_0$ and $g_0$ respectively, where $f_{0j}$, $0\leq j< r$, are either zero or homogeneous elements of $R$ with $\deg(f_{0j}) =l-ja$ and $g_{0i}$, $0\leq i< r$, are either zero or homogeneous elements of $R$ with $\deg(g_{0i}) =m-ia$. Suppose, if possible, that $\deg(fg)< l+m$. Then we would have $\deg(f_0g_0)< l+m$, i.e.,
		$$\sum\limits_{i+j=0}^{r-1}f_{0j}g_{0i}x^{i+j} + \sum\limits_{i+j=r}^{2r-2}f_{0j}g_{0i}(-\widehat{h})x^{i+j-r} = 0,$$
		i.e., $(\sum_{j=0}^{r-1}f_{0j}x^j)(\sum_{i=0}^{r-1}g_{0i}x^i)= 0$ in $\widehat{C}$ (we use the same symbol $x$ for the image of $X$ in $\widehat{C}$). Now as $\widehat{C}$ is an integral domain, it would follow that either $\sum_{j=0}^{r-1}f_{0j}x^j = 0$ or $\sum_{i=0}^{r-1}g_{0i}x^i = 0$, say $\sum_{i=0}^{r-1}g_{0i}x^i=0$ in $\widehat{C}$. Since $\{1,x,\dots, x^{r-1}\}$ is a basis of $\widehat{C}$ over $R$, it would follow that $g_0 = 0$ in $C$, a contradiction. Thus $\deg(fg) = l+m$.
		
		Since $R$ is a finitely generated $k$-domain and since $R$ is a graded ring with respect to the degree function $\deg$, by (\ref{r}), it follows that  this filtration on $C$ is admissible with respect to a homogeneous generating set of $R$ and $x$.  
		
		Let $\overline{C}$ denote the associated $\Z$-graded domain determined by the above filtration. For any $g\in C$, let $\overline{g}$ denote its image in $\overline{C}$. Note that $R$ can be regarded as a graded subring of $\overline{C}$.
		
		Let $\psi: R[X]\rightarrow \overline{C}$ be the surjective $R$-algebra homomorphism defined by $\psi(X)=\overline{x}$. Since $x^r+ \widehat{h} = -(h- \widehat{h})$ and since $\deg(h- \widehat{h}) <\deg (x^r+ \widehat{h})$, we have $\overline{x}^r + \widehat{h} = 0$ in $\overline{C}$. Since $\overline{C}$ is an integral domain with $\td_R\overline{C} =0$ and $(X^r + \widehat{h})$ is a prime ideal of $R[X]$, it follows that $\psi$ induces an isomorphism $\overline{\psi}:{R[X]}/{(X^r + \widehat{h})}\rightarrow \overline{C}$.
	\end{proof} 
	In particular, when $R= k^{[2]}$, we have the following result which will be needed subsequently.
	
	\begin{cor}\thlabel{graded}
		Let $r\in \N$ and $h(Y,Z)\in k[Y,Z]$
		be such that $X^r + h(Y,Z)$ is a prime element of $k[X,Y,Z]$. Suppose $k[Y,Z]$ has a $\pZ$-graded structure such that $X^r + \widehat{h}(Y,Z) $ is also a prime element of $k[X,Y,Z]$, where $\widehat{h}(Y,Z)$ is the highest degree homogeneous summand of $h$.
		Then we can define an admissible $\Z$-filtration on $$C:=\dfrac{k[X,Y,Z]}{(X^r + {h}(Y,Z)) }\text{~such that~} \gr(C)\cong \dfrac{k[X,Y,Z]}{(X^r + \widehat{h}(Y,Z))}.$$
	\end{cor}
	
	The next lemma has been proved by Crachiola and Maubach over a field of characteristic zero (\cite[Lemma 5.1]{CS}). We prove it for arbitrary characteristic.

	\begin{lem}\thlabel{Lemma 3}
		Let $a,b,c$ be three integers $\geq 2$ with $gcd(a,b,c) = 1$. Then \\$D \coloneq k[X,Y,Z]/(X^a + Y^bZ^c)$ is a rigid domain.
	\end{lem}
	\begin{proof}
		Let $x,y,z$ denote the images  of $X,Y$ and $Z$  in $D$ respectively. 
		Suppose, if possible, that there exists a  non-trivial exponential map $\phi$ on $D$ and let $f\in\D \setminus k$. 
		
		We first show that $\{x,y,z\}\cap D^\phi=\varnothing$.
		If $x\in D^{\phi}$ then $y,z\in D^{\phi}$, since $D^{\phi}$ is factorially closed (cf. \thref{lemma : properties}(i)), contradicting the fact that $\phi$ is non-trivial. So $x\notin D^{\phi}$. Next suppose, if possible, that $y\in\D$. Then by \thref{lemma : properties}(vi), $\phi$ would induce a non-trivial exponential map on $\widetilde{D} := k(y)[X,Z]/(X^a +y^bZ^c)$. 
		Then, denoting $L$ to be the algebraic closure of $k(y)$ in $\widetilde{D}$, by \thref{lemma : properties}(v), we would have $\widetilde{D}= L^{[1]}$. However,
		when $c\geq a$ then $x/z$ is in the integral closure of $\widetilde{D}$ in Frac$\widetilde{D}$ but not in $\widetilde{D}$ and when $c < a$ then $z/x$ plays the same role, showing that $\widetilde{D}$ is not a normal domain, in particular, $\widetilde{D}\neq L^{[1]}$, a contradiction. Therefore $y\notin D^{\phi}$. 
		Similarly we can conclude that $z\notin D^{\phi}$. 
		
		We consider two different $\Z$-gradings on the $k$-algebra $D$, given by
		\begin{center}
			$wt_1(x) = c,\, wt_1(y) = 0,\, wt_1(z) = a$\\
			$wt_2(x) = b,\, wt_2(y) = a,\, wt_2(z) = 0.$
		\end{center} 
		By \thref{Remark 1}, both the gradings define admissible $\Z$-filtrations on $D$. Now we are going to apply the gradings one by one. 
		For each $g\in D$, let $\widehat{g}$ denote the image of $g$ under the composite isomorphism $D\rightarrow gr_1(D)\rightarrow gr_2(gr_1(D))$.
		For $l = 1,2,$ let $\deg_l$ denote the degree function corresponding to the weight $wt_l$. By \thref{Theorem: homogeneization}, ${\phi}$ induces a non-trivial exponential map $\overline{\phi}$ on $D(\cong gr_1(D))$ with respect to $\deg_1$, and then $\overline{\phi}$ induces another non-trivial exponential map $\widehat{\phi}$ on $D(\cong gr_2(gr_1(D)))$ with respect to $\deg_2$, such that $gr_2(gr_1(f)) = \widehat{f}\in D^{\widehat{\phi}}.$ Then $\widehat{f}\notin k$ and with the help of the relation $x^a = -y^bz^c$, 
		$$\widehat{f}=\sum_{i\in\Gamma}h_i(y,z)x^i \text{~for some non-zero polynomials } h_i(y,z)\in k[y,z]$$
		with $\Gamma\subseteq\{ 0, 1,\dots,a-1\}$. 		 		
		
		We now show that $\Gamma$ is a singleton set.
		Suppose, if possible, there exist distinct $i,j\in\Gamma$. Then $\deg_l(h_i(y,z)x^i) =\deg_l(h_j(y,z)x^j)$ for $l= 1,2$. Let $\alpha_l = \deg_z(h_l)$ and $\beta_l = \deg_y(h_l)$ for $l = i,j$. Since 
		$\deg_1(h_i(y,z)x^i) =\deg_1(h_j(y,z)x^j)$,
		we would have
		\begin{equation}\label{5}
			\alpha_ia + ic = \alpha_ja + jc \text{~and hence~} (\alpha_i - \alpha_j)a = (j - i)c,
		\end{equation}
		and 
		$\deg_2({h_i(y,z)}x^i) = \deg_2({h_j(y,z)}x^j),$  would imply
		\begin{equation}\label{6}
			\beta_ia + ib = \beta_ja + jb \text{~and hence~} (\beta_i - \beta_j)a = (j - i)b.
		\end{equation}
		From $(\ref{5})$ and $(\ref{6})$ it would follow that there exists a prime factor $p_1$ of $a$ such that $p_1| c$ and $p_1|b$, and hence $p_1| gcd(a,b,c)$. But $gcd(a,b,c) = 1$, a contradiction. Thus $|\Gamma| = 1$. 
		
		Again since $x\notin D^{\widehat{\phi}}$ by the earlier argument, it would follow that $\widehat{f} = h_0{(y,z)}= \lambda y^nz^m$, for some $\lambda \in k^*,\, m,\,n\in \pZ$ with at least one of them strictly greater than 0. Therefore either $y\in D^{\widehat{\phi}}$ or $z\in D^{\widehat{\phi}}$ which is not possible. Thus there does not exist non-trivial exponential map $\phi$ on $D$, i.e., $D$ is rigid.	
	\end{proof}
	
	\begin{cor}\thlabel{Cor}
		Let $k$ be a field and $a,b,c\geq 2$ be integers such that $gcd(a,b,c)=1$. 
		Then for any $F(Y)\in k[Y]$, $D:=k[X,Y,Z]/(X^a + Y^bZ^c+ F(Y))$ is a rigid domain.
	\end{cor}
	\begin{proof}Let $x,y,z$ be the images of $X,Y$ and $Z$ in $D$ respectively.
		We define a $\pZ$-graded structure on $k[Y,Z]$, given by $wt(Y) = 0$ and $wt(Z) =a$.
		Then the highest degree homogeneous summand of $Y^bZ^c+ F(Y)$ is $Y^bZ^c$ and $X^a + Y^bZ^c$ is irreducible (as $gcd(a,b,c) = 1$) and hence prime in $k[X,Y,Z]$.
		Therefore, by \thref{graded}, the  degree function $\deg : D\setminus \{0\}\rightarrow \pZ$ defined by $$\deg(x) = c,\, \deg(y) =0, \,\deg(z) = a$$ admits an admissible $\Z$-filtration on $D$ such that $\gr(D)\cong k[X,Y,Z]/(X^a +Y^bZ^c)$. By \thref{Lemma 3}, $\gr(D)$ is rigid and hence, by \thref{Theorem: homogeneization}, $D$ is also rigid.
	\end{proof}
	
	\section{\textbf{Theorem B: Rigidity of the Pham-Brieskorn rings}}\label{ThmB}
	
	Recall that for a field $k$ of characteristic $p\geq 0$ and for any integer $n\geq 3$, $\underline{a}:=(a_1,a_2,\dots ,a_n)\in \N^n$ with $p\nmid gcd(a_1,\dots,a_n),$
	\begin{center}
		$B_{\underline{a}} = B_{(a_1,a_2,\dots ,a_n)} = \dfrac{k[X_1,X_2,\dots, X_n]}{(X_1^{a_1} + X_2^{a_2}+ \dots+X_n^{a_n})}$
	\end{center}
	denotes a Pham-Brieskorn domain. Note that $B_{(a_1,a_2,\dots ,a_n)} = B_{(a_{\sigma(1)},a_{\sigma(2)},\dots ,a_{\sigma(n)})}$ for every permutation ${\sigma}$ on $\{ 1,2,\dots, n\}$.
	Let $x_i$ be the image of $X_i$ in $B_{(a_1,a_2,\dots, a_n)}$, for $i\in\{1,2,\dots,n\}$. 
	Thus $B_{\underline{a}} = k[x_1,x_2,\dots, x_n]$. 
	Let $(d_1,d_2,\dots,d_n) = (L/a_1,L/a_2,\dots ,L/a_n),$ where $L= {\rm lcm} (a_1,a_2,\dots ,a_n)$. Then there is a unique $\pZ$-grading on $B_{\underline{a}}$ with the property that each $x_i$ is homogeneous of degree $d_i$, for all $i= 1,2,\dots ,n$. We call it the {\it standard $\pZ$-grading} on $B_{\underline{a}}$. The above mentioned grading is proper and admissible by \thref{Remark 1} and $\gr(\Ba)\cong \Ba$. For any $\phi\in$ EXP$(B_{\underline{a}})$, $\widehat{\phi}$ will denote its homogenization with respect to this standard $\pZ$-grading and for any $t\in B$, $\widehat{t}$ will denote its image in $\gr(B)(\cong B)$ with respect to the above mentioned grading. By \thref{Theorem: homogeneization}, if $\phi$ is non-trivial then $\widehat{\phi}$ is also non-trivial.

	\subsection{Proof of (i), (ii) and (iii) of Theorem B}\label{(I)}
	Let $k$ be a field of characteristic $p\geq 0$ and $n$ be an integer $\geq 3$, $\underline{a} := (a_1,a_2,\dots ,a_n)\in \pZ^n$. 
	\begin{enumerate}[\rm(i)]
		\item 
		
		\noindent
		Note that if $a_i = 1$, for some $i$, then $B_{\underline{a}}$ is a polynomial ring hence non-rigid.
		Suppose $a_1 = p^r, a_2 = sp^e$ for some $r,s,e\in\N$ and $r\leq e$. Then we have a non-trivial exponential map $\phi_T$ on $B_{\underline{a}}$ defined as follows:
		$$\begin{array}{llll}
			\phi_T(x_1) &=& x_1-((x_2 + T)^{sp^{e-r}} -x_2^{sp^{e-r}})\\
			\phi_T(x_2) &=& x_2 +T \\
			\phi_T(x_i) &=& x_i, \text{~for all~} i\geq 3.
		\end{array}$$	
		
		\item Suppose $a_1 = a_2 = 2$ and $k$ is a field containing a square root of $-1$; say, $i\in k$ is such that $i^2 = -1$. Then 
		$$\begin{array}{llll}
			B_{(2,2,a_3,\dots,a_n)}  &\cong & k[X_1,X_2,X_3,\dots, X_n]/(X_1^2 + X_2^2 + X_3^{a_3} +\dots+ X_n^{a_n})\\
			&\cong & k[U,V,X_3\dots, X_n]/(UV + X_3^{a_3} +\dots +X_n^{a_n})\\
			&\cong & k[u,v,x_3,\dots ,x_n], 	
		\end{array}$$
		where $U := (X_1 + iX_2), \, V := (X_1 - iX_2)$ and $u,v$ respectively denote the images of $U,V$ in $B_{\underline{a}}$.
		Then we have a non-trivial exponential map $\phi_T$ on $B_{\underline{a}}$ defined by,
		$$\begin{array}{lll} 
			\phi_T(u) &=& u\\
			\phi_T(v) &=& -[(x_3 + uT)^{a_3}+ \dots +(x_n + uT)^{a_n}]/u\\
			\phi_T(x_i) &=&x_i + uT, \text{~for all~} \, i\geq 3.
		\end{array}$$
		\item
		Let
		$A:= k[X_1,X_2,\dots,X_n]/(X_1^{a_1} + X_2^{a_2}+ \dots+X_n^{a_n})^m$ for some $m >1.$  
		Further let $x_i$ denote the image of $X_i$ in $A$ for all $i=1,2,\dots,n$. Then we have a non-trivial exponential map $\phi_T$ on $A$ defined as follows:
		\begin{enumerate}[\rm(a)]
			\item When $p\nmid \prod_{i} a_i$, then
			$$\begin{array}{lll}
				\phi_T(x_i)& =& x_i+({1}/{a_i})x_i(x_1^{a_1}+ x_2^{a_2}+ \dots + x_n^{a_n})^{m-1}T,\, 1\leq i\leq n.
			\end{array}$$
			
			\item When $a_n = sp^e$ for some $s,e\in \N$, $p\nmid s$ with $p^e\geq m> 1$ then,
			$$\begin{array}{lll}
				\phi_T(x_i)& =& x_i, \, 1\leq i\leq n-1\\
				\phi_T(x_n)& =& x_n+(x_1^{a_1}+ x_2^{a_2}+ \dots + x_n^{a_n})T.
			\end{array}$$
			
			\item When $a_n = sp^e$ for some $s,e\in \N$, $p\nmid s$ with $m>p^e> 1.$ Then there exist $r,\alpha\in \pZ$ such that $m= \alpha p^e + r$, $0\leq r< p^e$. Now
			$$\begin{array}{lll}
				\phi_T(x_i)& =& x_i,\, 1\leq i\leq n-1\\
				\phi_T(x_n)& =& x_n+(x_1^{a_1}+ x_2^{a_2}+ \dots + x_n^{a_n})^{\alpha +1}T.
			\end{array}$$ 
		\end{enumerate} 
	\end{enumerate}

	\subsection{Proof of (iv) of Theorem B}
	We first recall the definition of $F_3,\,R_3$, $T_3$ and $S_3$.
	\begin{flushleft}
		$F_3=\{(a_1,a_2,a_3)\in \N^3 \mid p\nmid gcd(a_1,a_2, a_3)\}$,\\
		$T_3 =\{(a_1,a_2, a_3)\in \N^3\mid a_i = 1$ for some $i$, or $\exists \,i,j\in \{1,2,3\},\, i\neq j,\, a_i = a_j =2\},$\\	
		$R_3=\{(a_1,a_2, a_3)\in \N^3\mid a_i = 1$ for some $i$, or  $\exists \, i,j\in \{1,2,3\},\, i\neq j,\, a_i=p^r,a_j= sp^e$, for $r,s,e\in \N ,\,r\leq e \}$\\
		and\\
		$S_3 =\{(2,2m,2p^e)\mid m, e \in \N, \, m>1 \text{~and~} p\nmid 2m \}.$ 
	\end{flushleft}
	Now from subsection \ref{(I)}, it follows that for an algebraically closed field $k$, if $B_{(a,b,c)}$ is rigid then $(a,b,c)\in F_3\setminus(R_3\cup T_3)$. We will show step by step that the converse is true except when $(a,b,c)\in S_3$.
	Let $x,y,z$ respectively denote the images of $X,Y,Z$ in $ k[X,Y,Z]/(X^a + Y^b + Z^c) = B_{(a,b,c)}$.
	
	The first proposition of this section is a useful observation, where rigidity of a subring  guarantees the rigidity of the Pham-Brieskorn domains under consideration.
	
	\begin{prop}
		\thlabel{gsubring criterion}
		Let $k$ be a field of characteristic $p>0$. For any integer $n\geq 3$ let $\underline{a} = (a_1,a_2,\dots, a_n)\in \N^n, s_2,\dots ,s_n, e_2, \dots ,e_n\in \N$ 
		be such that 
		\begin{enumerate}[\rm(i)]
			\item $e_2\leq e_3\leq\dots\leq e_n$
			\item $a_i = s_ip^{e^i}$, $2\leq i \leq n$
			\item $p\nmid a_1s_2\dots s_n$.
		\end{enumerate} 
		Then every non-trivial exponential map $\phi$ on the domain $B_{\underline{a}} \, (= k[x_1, x_2, \dots ,x_n])$ restricts to a non-trivial exponential map $\varphi$ on the subring $A\coloneq k[x_1, x_2^{p^r}, \dots ,x_n^{p^r}]$, where $r := e_2$. Moreover, $A^{\varphi} = A\cap \B_{\underline{a}}$. Thus $A$ rigid implies $B_{\underline{a}}$ is also rigid.
	\end{prop}
	\begin{proof}
		Let $y_i := x_i^{p^r}$ for $i\in \{2,\dots ,n\}$. Then $x_1^{a_1} + \sum_{i=2}^{n}y_i^{s_ip^{e_i-r}} =0$ in $B_{\underline{a}}$. Hence
		$$\begin{array}{lll}
			A = k[x_1,x_2^{p^r},\dots  ,x_n^{p^r}] & = & k[x_1,y_2,\dots ,y_n]\\
			& \cong & k[X_1,Y_2,Y_3,\dots ,Y_n]/(X_1^{a_1} + Y_2^{s_2} + Y_3^{s_3p^{e_3-r}} + \dots + Y_n^{s_np^{e_n-r}})\\
			& = & B_{(a_1,s_2,s_3p^{e_3 -r},\dots , s_np^{e_n-r})}.
		\end{array}$$
		It is well known that $A$ is normal (cf. \cite[Proposition 11.2]{Fossum}).		Suppose $\phi: B_{\underline{a}}\rightarrow B_{\underline{a}}[T]$ is an exponential map. Then it follows that $\phi(y_i) = \phi(x_i)^{p^r}\in A[T]$ for $i\in \{2,\dots ,n\}$. Let $K\coloneq$ Frac($A[T])$ and $L\coloneq$ Frac($B_{\underline{a}}[T])$. 
		Note that
		\begin{equation}\label{7}
			\phi(x_1)^{a_1} = -(\phi(x_2)^{s_2}+ \dots + \phi(x_n)^{s_np^{e_n-r}})^{p^r} = -(\phi(y_2)^{s_2} +\dots + \phi(y_n)^{s_np^{e_n - r}})\in A[T]\subseteq K. 	
		\end{equation}
		Next we observe that $L/K$ is a finite, purely inseparable extension and 
		$[L:K] = p^{(n-1)r}$. Now, as $\phi(x_1)\in L$, there exists $m\in \pZ$ such that 
		$\phi(x_1)^{p^m}\in K$. Since $gcd(a_1,p^m) = 1$ and $\phi(x_1)^{p^m}$,  
		$\phi(x_1)^{a_1}\in K$ (from (\ref{7})), it follows that $\phi(x_1)\in K$. 
		Therefore $\phi(x_1)\in K\cap B_{\underline{a}}[T] = A[T]$, since $A$ is a normal domain. Thus every exponential map $\phi$ on $B_{\underline{a}}$, restricts to an exponential map $\varphi$ on $A$ given by $\varphi(x_1) =\phi(x_1)$ and $\varphi(y_i) = \phi(x_i)^{p^r}$ for $2\leq i\leq n$. Hence $A^{\varphi} = A\cap \B_{\underline{a}}$.
	\end{proof}

	\begin{rem} {\rm The converse of the above proposition is not true in general, i.e., $A$ non-rigid may not imply that $B_{\underline{a}}$ is non-rigid. For example, over an algebraically closed field $k$ of characteristic $p>2$, $B_{(2,2p^r,2p^r)}$ is rigid for any $r\in\N$ (cf. \thref{a=2}) but $A = B_{(2,2,2)}$ is non-rigid.}
	\end{rem}
	
	\begin{rem}\thlabel{Remark 7}
		{\rm Suppose $(a,b,c)\in F_3\setminus T_3
			$ and $p\nmid abc$. By \thref{Theorem 1}, we know that $B_{(a,b,c)}$ is rigid if $(\frac{1}{a} + \frac{1}{b} + \frac{1}{c}) \leq 1$.
			Without loss of generality if we assume $a\leq b \leq c$,
			then the remaining cases we need to consider are $a =2,b = 3$ and $c\in\{3,4,5\}$. Note that in every such case, we get two pairs among $a,b,c$ which are co-prime.}
	\end{rem}
	
	We shall now discuss the rigidity of
	Pham-Brieskorn domains $B_{(a,b,c)}$ where two pairs among $a,b,c$ are co-prime. We first prove two lemmas in this regard. 
	
	\begin{lem}\thlabel{Lemma 2}
		Let $k$ be a field of characteristic $p\geq0$ and $(a,b,c)\in F_3\setminus R_3$, be such that $gcd(a,bc) = 1$. Then for any non-trivial exponential map $\phi$ on $B := B_{(a,b,c)} = k[x,y,z]$, $\{ x,y,z\}\cap \B = \varnothing$.
	\end{lem}
	\begin{proof}
		Since $gcd(a,b) = 1,~z$ is prime in $B$ and since $gcd(a,c) = 1$, $y$ is prime in $B$.
		Let $\phi$ be a non-trivial exponential map on $B$.
		
		Suppose, if possible, that  $x\in \B$. Then $y^b +z^c\in\B\setminus\{0\}$. We show that both $y$ and $z$ are in $\B$, which will contradict the fact that $\phi$ is non-trivial.\smallskip
		
		\noindent
		$\boldsymbol{Case \, I :}$ Suppose $p$ does not divide one of $b$ and $c$, say $p\nmid c$. Then, by \thref{Mini Mason 1} or \thref{Mini Mason 2}, according as $p\nmid b$ or $p\mid b$, $y,z\in\B$.
		\smallskip
		
		\noindent
		$\boldsymbol{Case \, II :}$
		Suppose $p$ divides both $b$ and $c$. 
		Let $b = s_2p^r$ and $ c = s_3p^e$ for some $s_2,s_3,e,r\in \N$ with $e \geq r$ and $p\nmid s_2s_3$. Now $y^b +z^c = (y^{s_2} + z^{s_3p^{e-r}})^{p^r}\in\B\setminus\{0\}$, hence by \thref{lemma : properties}(i), $y^{s_2} + z^{s_3p^{e-r}}\in \B\setminus\{0\}$. If $e>r$ then $s_2 > 1$, for $(a,s_2p^r,s_3p^e)\notin R_3$, and hence by \thref{Mini Mason 2}, $y,z\in\B$. Next when $e=r$ then
		$s_2,s_3\geq 2$, since $(a,s_2p^r,s_3p^e)\notin R_3$. Then  by \thref{Mini Mason 1}, $y,z\in\B$. 
		
		Hence $x \notin B^{\phi}$. Now suppose, if possible, that $y\in\B$. Then $x^a + z^c\in \B\setminus\{0\}$.
		
		If $p\nmid c$ then, by \thref{Mini Mason 1} or \thref{Mini Mason 2}, according as $p\nmid a$ or $p\mid a$, $x,z\in\B$. But $\phi$ is non-trivial hence $p\mid c$. Then $p\nmid a$ since $gcd(a,c) = 1$.
		Let $c = s_3p^e$ for some $s_3,e\in \N$ with $p\nmid s_3$. By \thref{lemma : properties}(vi) and \thref{Remark 2}(i), $\phi$ induces a non-trivial exponential map  on the integral domain $\widetilde{B} := B\otimes_{k[y]}\overline{k(y)} = \overline{k(y)}[X,Z]/(X^a + Z^c +y^b).$ 
		Since $\td_{\overline{k(y)}}{\widetilde{B}} = 1$, by \thref{lemma : properties}(v) $\widetilde{B}=
		\overline{k(y)}^{[1]}$.
		Now consider the irreducible polynomial,
		$$\begin{array}{lll}
			F(X,Z) & := & X^a + Z^c + y^b \in \overline{k(y)}[X,Z]\\
			& = & X^a + (Z^{s_3} - \beta)^{p^e}, \text{~for some~}\beta\in \overline{k(y)}.
		\end{array}$$
		Let  $\gamma$ be a root of $Z^{s_3} - \beta$. Then $M := (X, Z-\gamma)$ is a maximal ideal of $\overline{k(y)}[X,Z]$ and $F\in M^2$. Thus $\widetilde{B}$ is not a normal domain, in particular $\widetilde{B}\neq
		\overline{k(y)}^{[1]}$. This leads to a contradiction.   
		
		Similarly we can show that $z\notin \B$. Thus our lemma follows. 
	\end{proof} 
	
	\begin{lem}{\thlabel{model lemma}}
		Let $k$ be an algebraically closed field of characteristic $p\geq0$ and $(a,b,c)\in F_3\setminus R_3$ with $gcd(a,bc) = 1$. Then for any non-trivial exponential map $\phi$ on $B := B_{(a,b,c)} = k[x,y,z]$, $y^{b/d} +\mu z^{c/d}\in B^{\widehat{\phi}}$ for some $\mu\in k^*$ and $d = gcd(b,c)$ {\em (for the meaning of $\widehat{\phi}$ please refer to beginning of this section)}.
	\end{lem}
	\begin{proof}
		Let $\phi$ be a non-trivial exponential map on $B$ and let $f\in \B \setminus k$. 
		Using the relation $x^a = -(y^b + z^c)$, $f$ can be uniquely written as 
		\begin{center}
			$f = \sum\limits_{i=0}^{a-1}f_i(y,z)x^i$ for some polynomials $f_i(y,z)\in k[y,z]\, $ in $B$.
		\end{center}
		By \thref{Lemma 2} $x\notin\B$. Therefore $f_0(y,z)\notin k$. Now we consider the standard $\pZ$-graded structure on $B$, given by 
		\begin{center}
			$wt(x) = bc/d$,  $wt(y) = ac/d$,  $wt(z)= ab/d$.
		\end{center}
		Let $\deg$ denote the degree function corresponding to the weight $wt$.
		
		If $\deg(x^{i_1}y^{j_1}z^{l_1}) = \deg(x^{i_2}y^{j_2}z^{l_2})$ for some $i_1,i_2,j_1,j_2,l_1,l_2\in \pZ$ with $0\leq i_1< i_2< a$. Then $((j_1 - j_2)(c/d) + (l_1 - l_2)(b/d))a = (i_2 - i_1)(bc/d)$, which implies that $a|(i_2 - i_1)(bc/d)$. But $gcd(a,bc) = 1$ and $0< i_2 - i_1< a$. Therefore two different powers of $x$ cannot occur in ${\widehat{f}}$.
		
		Thus $\widehat{f} = \widehat{f_i(y,z)}x^i $ for some $i,\, 0\leq i < a$. Now by {\thref{Theorem: homogeneization}}, $\widehat{\phi}$ is non-trivial and $\widehat{f}\in B^{\widehat{\phi}}$. By \thref{Lemma 2} $x\notin \Bh$, therefore
		$\widehat{f}=\widehat{f_0(y,z)}$. 
		Next, $k$ is an algebraically closed field, so $\widehat{f} = \lambda y^rz^m \prod_{l\in \Lambda}^{}(y^{b/d} + \mu_lz^{c/d})$ for some $\lambda,\mu_l \in k^*$, $r,m\in \pZ$ and a finite set $\Lambda\subseteq \pZ$, in $B$.
		
		By \thref{Lemma 2} $y,z\notin \Bh$ so $r = m=0$. If $\widehat{f} $ has two distinct factors, $y^{b/d} + \mu_1 z^{c/d}$ and $y^{b/d} + \mu_2 z^{c/d}$, then $z\in\Bh$, a contradiction. Thus $\widehat{f} = \lambda(y^{b/d} + \mu z^{c/d})^r$ for some $\lambda,\mu\in k^*$ and $r\in\N$. Hence $y^{b/d} + \mu z^{c/d}\in B^{\widehat{\phi}}$ (cf. \thref{lemma : properties}(i)). 
	\end{proof}
	\begin{thm}\thlabel{Theorem 2}
		Let $k$ be a field of characteristic $p\geq 0$ and $(a,b,c)\in F_3\setminus (R_3\cup T_3)$ with $gcd(a,bc) = 1$. Then $B := B_{(a,b,c)} = k[x,y,z]$ is a rigid domain. 
	\end{thm}
	\begin{proof}
		In view of \thref{Remark 2}(i), we assume $k$ to be an algebraically closed field. 	
		Suppose, if possible, that there exists a non-trivial exponential map $\varphi$ on $B$.
		By \thref{model lemma}, $y^{b/d} + \mu z^{c/d}\in B^{\widehat{\varphi}}\setminus \{0\}$ for some $\mu\in k^{*}$ and
		$d= gcd(b,c)$. Note that $gcd(a,bc) = 1$ so $y,z$ are prime in $B$.
		For convenience let $b\leq c$.
		Next we consider different cases and arrive at contradiction in each of them and hence it will follow that $B$ is a rigid domain.
		\smallskip
		\noindent
		$\boldsymbol{Case \, I :} d < b <c$.\\
		Then $b/d,c/d \geq 2$ 
		and since $y,z$ are primes in $B$, by {\thref{Mini Mason 1}} or {\thref{Mini Mason 2}} it follows that $y,z\in B^{\widehat{\varphi}}$. Thus $\widehat{\varphi}$ is a trivial exponential map, a contradiction.\smallskip 
		
		\noindent
		$\boldsymbol{Case\,II :}d = b < c.$\\
		Then $y + \mu z^{\frac{c}{b}} \in B^{\widehat{\varphi}}$ for some $\mu \in k^*$. Now for every $\lambda$ in $k^{*}$ 
		$$
		\frac{B}{(y+\mu z^{c/b} - \lambda)B} = \frac{k[X,Y,Z]}{(X^a + Y^b + Z^c, Y+\mu Z^{c/b} -\lambda )}\cong \frac{k[X,Z]}{(X^a + P_{\lambda}(Z^{c/b}))},
		$$
		where 		\begin{equation}
			P_{\lambda}(T) := (1 + (-\mu)^b)T^b + {b\choose {b-1}}\lambda (-\mu T)^{b-1}+\dots + {b\choose 1}\lambda^{b-1}(-\mu T) + \lambda^b\in k[T].	
		\end{equation}
		Then $P_{\lambda}(T)\in k^{*}$ if and only if $1+(-\mu)^b =0$ and $b = p^r,$ for some $r\in \N.$ But $(a,b,c)\notin R_3$ and hence $y + \mu z^{c/b} - \lambda$ is a prime element of $B$ for all $\lambda\in k^{*}.$ 
		Therefore  by \thref{betta-lemma}, there exists  a $\beta\in k^*$ such that $\widehat{\varphi}$ induces a non-trivial exponential map on the integral domain $\widetilde{B}:= B/(y+\mu z^{c/b} - \beta)B$.
		Now
		$$\widetilde{B} \cong \frac{k[X,Z]}{(X^a + P_{\beta}(Z^{c/b}))}.$$
		Since $\td_k\widetilde{B}$ = 1 hence by \thref{Remark 2}(ii), $\widetilde{B} = k^{[1]}$. Thus $P_{\beta}(Z^{c/b})$ should have a linear term in $Z$. This is not possible as $c/b >1$.\smallskip 
		
		\noindent
		$\boldsymbol{Case \, III:} d = b = c$ and $p| b$.\\
		As in Case II, we get $\mu,\, \beta \in k^*$ such that $\widehat{\varphi}$ induces a non-trivial exponential map on the integral domain $\widetilde{B}:= B/(y+\mu z - \beta)B$ and
		$$\widetilde{B} \cong \frac{k[X,Z]}{(X^a + P_{\beta}(Z))}= k^{[1]}.$$
		
		\noindent
		But $p|b$ hence $P_{\beta}(Z)$ does not have a linear term. Therefore, $\widetilde{B} \neq k^{[1]}$  a contradiction.\smallskip
		
		\noindent
		$\boldsymbol{Case \, IV:} d = b =c,\, p\nmid b$ and $p|a$.\\
		Following Case II, we get $\mu,\, \beta\in k^{*}$ such that $y+\mu z\in B^{\widehat{\varphi}}$ and $$\widetilde{B} := B/(y+\mu z - \beta)B \cong \frac{k[X,Z]}{(X^a + P_{\beta}(Z))}= k^{[1]}.$$
		Now $b = c \geq 3$ (as $(a,b,c)\not\in T_3$) and $p\nmid b$ therefore $\deg_ZP_{\beta}(Z) \geq 2$. When $p\nmid (b-1)$ then the coefficient of $Z^{b-2}$ in $P_{\beta}^{'}(Z)$ is non-zero. When $b = sp^r +1$, for some $s,r\in\N$ with $s>1$ and $p\nmid s $ then the coefficient of $Z^{(s-1)p^r}$ in $P_{\beta}^{'}(Z)$ is non-zero. Thus $\deg_ZP_{\beta}'(Z) \geq 1$ except when $1+(-\mu)^b = 0$ and $b=c=p^r + 1$, for some $r\in\N$. Now when $\deg_ZP_{\beta}'(Z) \geq 1$ we can find $a_1, a_2\in k$ such that $P_{\beta}'(a_2) = 0$ and $a_1^a + P_{\beta}(a_2) = 0$. Therefore by Jacobian criterion $\widetilde{B}$ is not normal hence cannot be $k^{[1]}$, a contradiction.
		
		Let us now consider the case when  $1 + (-\mu)^b = 0$ and $b=c=p^r + 1$ for some $r\in\N$. Let $a= s_1p^m$ for some $s_1,m\in \N$ and $p   \nmid s_1$. Now when $p>2$ and $s_1>1$, then $1/s_1+1/(p^r+1) +1/(p^r+1) \leq \frac{1}{2} + \frac{1}{4} + \frac{1}{4} = 1$ and when $p=2$ with $s_1\geq 3$ then $1/s_1+1/(p^r+1) +1/(p^r+1) \leq \frac{1}{3} + \frac{1}{3} + \frac{1}{3} = 1$, hence by \thref{Stably rigid 1}, $B$ is rigid. Thus the only remaining case is when $s_1 = 1$. Suppose $s_1=1$. 
		
		Let $Z_1= -\mu Z$ and $Y_1= Y + \mu Z = Y - Z_1$, then
		$$\begin{array}{lll}
			X^{p^m} + Y^{p^r +1} + Z^{p^r + 1} &= &X^{p^m} + (Y_1+  Z_1)^{p^r + 1} - (Z_1)^{p^r +1}, \text{~since~}(-\mu)^{p^r +1} = -1\\					
			&= &X^{p^m} + Y_1^{p^r + 1} + Y_1^{p^r}Z_1 + Y_1 Z_1^{p^r}. \\ 
		\end{array}$$
		Therefore $B =B_{(p^m,p^r+1,p^r+1)} \cong k[X,Y_1,Z_1]/(X^{p^m} + Y_1^{p^r + 1} + Y_1^{p^r}Z_1 + Y_1 Z_1^{p^r})=k[x,y_1,z_1] $, where $x,y_1 (=y+\mu z),z_1 (= -\mu z)$ denote the images of $X,Y_1,Z_1$ in $B$  respectively. Now $y_1\in B^{\widehat{\varphi}}$ hence by \thref{lemma : properties}(vi), $\widehat{\varphi}$ induces a non-trivial exponential map on $B_1:= k(y_1)[X,Z_1]/(X^{p^m} +  y_1 Z_1^{p^r} + y_1^{p^r}Z_1 + y_1^{p^r + 1}).$ Next $B_1\otimes_{k(y_1)}\overline{k(y_1)} = \overline{k(y_1)}^{[1]}$, by \cite[Proposition 4.2]{As} and $\td_{k(y_1)}B_1=1$. Therefore by \thref{Remark 2}(ii), $B_1 = {k(y_1)}^{[1]}$. But according to \cite[Remark 4.5]{As}, $B_1$ is a non-trivial $\mathbb{A}^1$-form over $k(y_1)$ (i.e. $B_1\neq k(y_1)^{[1]}$). This is a contradiction.\smallskip
		
		\noindent
		$\boldsymbol{Case \, V:} d = b =c,\, p\nmid ab$.\\
		By \thref{Stably rigid 1}, if $(\frac{1}{a} + \frac{1}{b} + \frac{1}{c}) \leq 1$ then $B$ is rigid.
		So the remaining case is when $a =2$ and $b = c =3$.
		As in Case II, we get $\mu,\, \beta \in k^*$ such that $$\widetilde{B}:=  B/(y+\mu z -\beta)\cong \frac{k[X,Z]}{(X^2 + P_{\beta}(Z))}= k^{[1]},$$
		with $P_{\beta}(Z) = (1  -\mu^3)Z^3 + 3\beta \mu^2 Z^2- 3\beta^{2}\mu Z + \beta^3\in k[Z]$. 
		
		If $1 - \mu^3 \neq 0$, then $\deg_Z(P_{\beta}(Z)) = 3$. Now since $\widetilde{B} = k^{[1]}$, there exists a surjective $k$-algebra homomorphism, $\varPsi:k[X,Z]\rightarrow k[T](= k^{[1]})$ such that $\deg_T(\varPsi(X)) =\deg_Z(P_{\beta}(Z)) = 3$ and $\deg_T(\varPsi(Z)) = 2$. So all the conditions of Epimorphism theorem (cf. \thref{EpiTheorem}) are satisfied. But $gcd(2,3) = 1$, which leads to a contradiction. So $\mu^3 = 1$. 
		Therefore $X^2 + P_{\beta}(Z) = X^2 + 3\beta \mu^2 Z^2- 3\beta^{2}\mu Z + \beta^3 = X^2 + 3\beta(\mu Z -\beta/2 )^2 + \beta^3/4$. Then $X^2 + P_{\beta}(Z)$ has no linear term in the new co-ordinates $X$ and $Z_1:=(\mu Z -\beta/2)$  and hence $\widetilde{B}\neq k^{[1]}$. This leads to a contradiction.
	\end{proof}
	
	Now we consider all those Pham-Brieskorn domains $ B_{(a,b,c)}$ over a field $k$ of characteristic $p\geq 0$ such that $p\nmid abc$. 
	
	\begin{thm}\thlabel{p not dividing}
		Let $k$ be a field of characteristic $p\geq 0$. Then $B_{(a,b,c)}$  is rigid when $(a,b,c)\in F_3\setminus T_3$ with $p\nmid abc$.
	\end{thm}
	\begin{proof}
		By \thref{Remark 7} and \thref{Theorem 2} our theorem follows.	
	\end{proof}	
	
	Next we turn our attention to all those Pham-Brieskorn domains $B_{(a,b,c)}$ over a field $k$ of characteristic $p>0$, where only one among $a,b,c$ has $p$ as a factor. 
	
	\begin{thm}\thlabel{Theorem 3}
		Let $k$ be a field of characteristic $p>0$ and $B := B_{(a,b,c)}$, for $c = sp^e$ for some $a,b,s,e\in \N$ with $p\nmid abs$.
		Then $B$ is rigid when $(a,b,c)\notin (T_3\cup S_3)$. 
	\end{thm}
	\begin{proof}	
		If $s = 1$, then $gcd(ab,c) = 1$  and hence by \thref{Theorem 2}, $B$ is rigid.
		So henceforth we assume $s>1$.
		
		If $(\frac{1}{a}+\frac{1}{b} + \frac{1}{s})\leq 1$ then by \thref{Stably rigid 1}, we know that $B$ is rigid.
		
		\noindent
		We consider below the remaining cases. For convenience let $a\leq b$.
		
		\medskip
		
		\noindent
		$\boldsymbol{(a)}$ $\{a,b,s\} = \{ 2,3,5\}$ or $\{2,3,4\}$ or $\{2,3,3\}$. Then as $p\nmid abs$,
		in every case we get two pairs among $a,b,c$, which are co-prime. Therefore by \thref{Theorem 2}, $B$ is rigid.
		
		\smallskip
		
		\noindent
		$\boldsymbol{(b)}$ Otherwise $ a = 2 = s$ and $b $ is an odd integer greater than $1$ ($(a,b,c)\notin S_3$) i.e. $(a,b,c) = (2,2m+1,2p^e)$ for some $m\in \N$. Then $gcd(ac,b) = 1$ since $2\nmid b$ and $p\nmid b$. The result again follows from \thref{Theorem 2}.
	\end{proof}
	
	Next we consider all those Pham-Brieskorn domains $B := B_{(a,b,c)}$ over a field $k$ of characteristic $p> 0$ such that $p$ occurs as factors in $b$ and $c$.
	Let $b = s_2p^r$ and $c= s_3p^e$ for some $s_2,s_3,e,r\in\N$, $p\nmid as_2s_3$ and $e\geq r$.
	
	\begin{lem}\thlabel{Remark 5}	
		Suppose $\phi$ is an exponential map on $B = k[x,y,z]$. 
		Then $x,y^{s_2} + z^{s_3p^{e-r}}\in \B$.
	\end{lem}		
	\begin{proof}
		Note that if $\phi_1:=\phi\otimes id$ is the extension of $\phi$ on $B\otimes_k{\overline{k}}$ and $x,y^{s_2} + z^{s_3p^{e-r}}\in (B\otimes_k\overline{k})^{\phi_1}$, then $x,y^{s_2}+ z^{s_3p^{e-r}}\in (B\otimes_k\overline{k})^{\phi_1}\cap B= \B.$ 
		Therefore we assume henceforth $k$ is an algebraically closed field.
		
		Let $\phi$ be a non-trivial exponential map on $B$ and $f\in B^{\phi}\setminus k$. Let $\p = (x,h(y,z))B$, where $h(y,z)$ is an irreducible factor of $y^{s_2}+ z^{s_3p^{e-r}}$. 
		Then $\p$ is a prime ideal of $B$ of height one and $B_{\p}$ is not a normal domain.
		Now by \thref{lemma : properties}(vi), $\phi$ induces a non-trivial exponential map on $B_1 := B\otimes_{k[f]}k(f)$. 
		Since $\td_{k(f)}B_1 = 1$, by \thref{lemma : properties}(v), $B_1 = L^{[1]}$, where $L$ is the algebraic closure of $k(f)$ in $B_1$. 
		Hence $B_1$ is a normal domain. 
		Therefore $\p B_1 = B_1,$ i.e. $\p \cap k[f]\neq (0)$. 
		Hence $\p\cap k[f] = (f-\lambda_{\p})k[f]$ for some $\lambda_{\p}\in k.$ 
		Let $f = \sum_{i=0}^{a-1}f_i(y,z)x^i$ for some polynomials $f_i(y,z)\in k[y,z]$. 
		Then as $f-\lambda_{\p}\in \p$, it follows that $\theta_{\p}:= f_0 -\lambda_{\p}\in (h)B.$ 
		Thus if $h_1,h_2,\dots, h_n$ are  all the distinct irreducible factors of $y^{s_2}+ z^{s_3p^{e-r}}$ for some $n\in \N$ then we obtain $\theta_{\p_1}, \theta_{\p_2}, \dots,\theta_{\p_n}\in k[y,z]$ such that $\prod_{i=1}^{n}\theta_{\p_i}\in (h_1h_2\dots h_n)B $. 
		Therefore there exists an integer $m >0$ such that 
		\begin{equation}\label{9}
			(\prod_{i=1}^{n} f_0-\lambda_{\p_i})^m =(\prod_{i=1}^{n}\theta_{\p_i})^{m}\in ((y^{s_2} + z^{s_3p^{e-r}})^{p^r})B\subseteq (x)B.
		\end{equation}
		Hence $(\prod_{i=1}^{n} f-\lambda_{\p_i})^{m}\in (x)B$, by (\ref{9}). 
		Since  $(\prod_{i=1}^{n} (f-\lambda_{\p_i}))^{m}\in \B$ so $x\in\B$ (cf. \thref{lemma : properties}(i)) and hence $y^{s_2} + z^{s_3p^{e-r}}\in \B$.	
	\end{proof}	
	
	\begin{lem}\thlabel{a=2}
		Let $k$ be a field of characteristic $p>0$. 
		Let $B_1 = B_{(2,2mp^r,s_3p^e)}$, where $m,s_3,r,e\in\N$, $e\geq r$,  $p\nmid 2ms_3$ and $(2,2mp^r,s_3p^e)\in F_3\setminus R_3.$ 
		Then $B_1$ is rigid.
	\end{lem}
	\begin{proof}
		Without loss of generality, we assume, $k$ to be an algebraically closed field (cf. \thref{Remark 2}(i)).
		Let $\phi$ be an exponential map on $B_1$. Then by \thref{Remark 5}, $x, y^{2m} + z^{s_3p^{e-r}}\in \B_1$.
		We will show that $y,z\in \B_1$.
		Then $\phi$ is a trivial exponential map. 
		So $B_1$ is rigid.
		\smallskip
		
		\noindent
		$\boldsymbol{Case \,I:} s_3$ is odd.\\
		Note that if $s_3=1$ then $e>r$ as $(2,2mp^r,s_3p^e)\in F_3\setminus R_3$. Now $gcd(2,s_3p^e) = 1$ hence $y$ is prime in $B_1$.
		Since $y^{2m} + z^{s_3p^{e-r}}\in B_1^{\phi}$, it follows that $y,z\in B_1^{\phi}$, by \thref{Mini Mason 1} or \thref{Mini Mason 2}, according as $e=r$ or $e>r$.
		\smallskip
		
		\noindent 
		$\boldsymbol{Case \,II:} s_3$ is even,
		say $s_3 = 2s_4$, for some $s_4\in\N$ and $i\in k$ be such that $i^2 = -1$.\\
		Now $y^{2m} + z^{2s_4p^{e-r}}\in B_1^{\phi}$. 
		Hence $y^{m} + iz^{s_4p^{e-r}},\, y^{m} - iz^{s_4p^{e-r}}\in B_1^{\phi}$ (cf. \thref{lemma : properties}(i)). 
		Therefore $y,z\in B_1^{\phi}$. 
	\end{proof}
	
	\begin{thm}\thlabel{Prop 3}
		Let $k$ be a field of characteristic $p>0$, $B:= B_{(a,s_2p^r,s_3p^e)}$, where $s_2, s_3,e,r\in \N,\, p\nmid as_2s_3$ and $ e\geq r$. 
		Then $B$ is a rigid domain whenever $(a,s_2p^r,s_3p^e)\in F_3\setminus R_3.$ 
	\end{thm}
	\begin{proof}
		Without loss of generality, we assume, $k$ to be an algebraically closed field (cf. \thref{Remark 2}(i)). 
		We consider the subring  $A \coloneq k[x,y^{p^r},z^{p^r}] = B_{(a,s_2, s_3p^{e-r})}$ of $B$. 
		Based on $e,r$ we consider two cases:
		\smallskip
		
		\noindent
		$\boldsymbol{Case \, I: } $ Let $e=r$. Since $(a,s_2p^r,s_3p^r)\in F_3\setminus R_3$, $s_2,s_3\geq 2$. For convenience let $s_2\leq s_3$.\\
		When at most one of $a,s_2, s_3$ is equal to $2$, then  $(a,s_2,s_3)\in F_3\setminus T_3$ and given, $p\nmid as_2s_3$ therefore by \thref{p not dividing}, $A$ is rigid. Hence by \thref{gsubring criterion} $B$ is also rigid.
		If $a=2$ and $s_2 = 2$ then $B = B_{(2,2p^r,s_3p^r)}$ is rigid by \thref{a=2}. Now the remaining case is when $a>2$ and $s_2 = s_3 =2$. 
		Then 
		$$
		B = B_{(a,2p^r,2p^r)}= \dfrac{k[X,Y,Z]}{(X^a +(Y^2 + Z^2)^{p^r})}\\
		\cong \dfrac{k[X,U,V]}{(X^a + (UV)^{p^r})},
		$$
		where $U:= Y+iZ$ and $V:= Y-iZ$, for some $i\in k$ with $i^2 =-1$.
		Since $gcd(a,p^r,p^r) = 1$ and $a,p^r\geq 2$, by \thref{Lemma 3}, $B$ is rigid.
		\smallskip
		
		\noindent
		$\boldsymbol{Case \, II: }$ Let $e>r$ then $s_2\geq 2.$\\
		When $a>2$ or $s_2>2$ and $(a,s_2,s_3p^{e-r})\in F_3\setminus (T_3\cup S_3)$, then $A$ is rigid by \thref{Theorem 3} and hence by \thref{gsubring criterion} $B$ is also rigid.
		Otherwise $a=2$ and $s_2$ is even. 
		Hence by \thref{a=2}, $B$ is rigid.
	\end{proof}
	
	Now Theorem B(iv) follows.
	
	\begin{thm}\thlabel{main}
		$B_{(a,b,c)}$ is rigid when $(a,b,c)\in F_3\setminus(R_3\cup T_3\cup S_3)$.
	\end{thm}
	\begin{proof}
		By combining Theorems \ref{p not dividing}, \ref{Theorem 3} and \ref{Prop 3} the theorem follows.
	\end{proof}
	
	\section{\textbf{Applications}}\label{app}
	As an application of Theorem B and results of section 4, we now prove rigidity of certain surfaces and threefolds.  
	We first consider a threefold which the first author encountered in her paper \cite{G1}.
	
	\begin{prop}\thlabel{prop 4}	
		Let k be a field of characteristic $p>0$. Let
		$$ 
		B =\dfrac{k[X,Y,Z,T]}{(X^mY^n+T^{qp^r} + Z^{p^e})},
		\text{~for some~} m,n,q,r,e\in \N 
		\text{~with~}  p\nmid q
		\text{~and~} e>r\geq 1.
		$$ 
		Let $x,y,z,t$  respectively denote the  images of $X,Y,Z,T$ in $B$. 
		Then the following statements are true:
		\begin{enumerate}[{\rm (i)}]
			
			\item If $p\nmid m$ and $m, q \ge 2$,
			then there does not exist any non-trivial exponential map $\phi$ on $B$ such that $y\in \B$.
			\item If $m,n\geq 2$ and $gcd(m,n,p) =1$, then there does not exist any non-trivial exponential map $\phi$ on $B$ such that $t\in \B$.
			
			\item  If $m,n\geq 2$ and $gcd(m,n,pq) = 1$, then there does not exist any non-trivial exponential map $\phi$ on $B$ such that $z\in \B.$
		\end{enumerate} 
	\end{prop}
	\begin{proof}
		Let $\phi$ be a non-trivial exponential map on $B$.
		\medskip
		
		\noindent
		$\boldsymbol{(i)}$
		Suppose, if possible, that $y\in \B$. Then, by \thref{lemma : properties}(vi) and \thref{Remark 2}(i), $\phi$ induces a non-trivial exponential map on $\widetilde{B}:=\overline{k(y)}[X,Z,T]/(y^nX^m + T^{qp^r} + Z^{p^e})$.
		Since $e >r$, $q>1$ and $m>1$, we have $(m, qp^r, p^e)\in F_3\setminus R_3$. Hence, 
		by \thref{Prop 3}, $\widetilde{B}$ is rigid, a contradiction.
		\smallskip
		
		\noindent
		$\boldsymbol{(ii)}$
		Suppose, if possible, that $t\in \B$. Then  again by \thref{lemma : properties}(vi), $\phi$ induces a non-trivial exponential map on $\widetilde{B}:=k(t)[X,Y,Z]/(X^mY^n + Z^{p^e}+ t^{qp^r})$.
		Since $m,n,p^e\geq 2$ and $gcd(m,n,p^e) = 1$, by \thref{Cor}, $\widetilde{B}$ is rigid, a contradiction.
		\smallskip
		
		\noindent
		$\boldsymbol{(iii)}$
		Suppose, if possible, that $z\in \B$. Then, as before, 
		we get a non-trivial exponential map on $\widetilde{B}:=k(z)[X,Y,T]/(X^mY^n + T^{qp^r}+ z^{p^e})$. 
		Now $m,n,qp^r\geq 2$ and $gcd(m,n,qp^r) = 1$, hence by \thref{Cor}, $\widetilde{B}$ is rigid, a contradiction.
	\end{proof}
	
	\begin{rem}
		{\em  \thref{prop 4}(i) provides an alternative proof of \cite[Lemma 3.5]{G1} in some cases.}
	\end{rem} 
	
	The next result relates the rigidity of Pham-Brieskorn surfaces with its translates.
	
	\begin{prop}\thlabel{Lemma 5}
		Let $A = k[X,Y,Z]/(X^a + Y^b + Z^c + \lambda)$ for $(a,b,c)\in F_3$ and $\lambda\in k$. 
		Then $A$ is a rigid domain whenever $B_{(a,b,c)}$ is rigid.	
	\end{prop}
	\begin{proof}
		Let $x,y,z$ be the images of $X,Y$ and $Z$ in $A$ respectively. 
		We define a $\pZ$-graded structure on $k[Y,Z]$, given by $wt(Y) = ac$ and $wt(Z) =ab$.
		Then the highest degree homogeneous summand of $Y^b + Z^c + \lambda$ is $Y^b+Z^c$.
		Now $X^a + Y^b + Z^c$ is a prime element of $k[X,Y,Z]$ because $(a,b,c)\in F_3$.
		Therefore by \thref{graded}, the function $\deg : A\setminus \{0\}\rightarrow \pZ$ defined by 
		$$
		\deg(x) = bc,\, \deg(y) =ac,\, \deg(z) = ab
		$$ 
		admits an admissible $\Z$-filtration on $A$ such that 
		$$
		\gr(A)\cong \dfrac{k[X,Y,Z]}{(X^a + Y^b + Z^c)} = B_{(a,b,c)}.
		$$  
		By \thref{Theorem: homogeneization}, every non-trivial exponential map on  
		$A$ induces a non-trivial exponential map on $gr(A)=B_{(a,b,c)}$, therefore $A$ is rigid whenever $B_{(a,b,c)}$ is rigid. 	
	\end{proof}
	
	As we have seen in  \cite{G1}, the non existence of non-trivial exponential maps turns out to be crucial in settling the triviality or otherwise of important affine domains.
	Below we give a few sufficient conditions for non existence of certain non-trivial exponential maps on certain Pham-Brieskorn threefolds.
	
	\begin{prop}
		Let $k$ be a field of characteristic $p\geq 0$. 
		For $(a,b,c,d)\in F_4\setminus(R_4\cup T_4)$ let $B=k[X,Y,Z,T]/(X^a+Y^b+Z^c+T^d)=B_{(a,b,c,d)}= k[x,y,z,t]$ {\rm (where $x,y,z,t$ respectively denote the images of $X,Y,Z$ and $T$ in $B$)}. 
		Then for the following conditions on $b,c,d$, there does not exist any non-trivial exponential map on $B$ fixing $x$.
		\begin{enumerate}[{\rm (i)}]
			
			\item $(b,c,d)\in F_3\setminus  S_3$.
			
			\item $b=s_2p^m$, $c = s_3p^r$ and $d = s_4p^e$ for some $s_2,s_3,s_4,m,r,e\in \N$ with $p\nmid s_2s_3s_4,\, m\leq r\leq e$ and $(s_2,s_3p^{r-m}, s_4p^{e-m})\in F_3\setminus (T_3\cup S_3).$
		\end{enumerate}		
	\end{prop}
	\begin{proof}
		Suppose, if possible, $\phi$ is a non-trivial exponential map on $B$ with $x\in \B$.
		We arrive at a contraction in each case.
		\smallskip
		
		\noindent
		$\boldsymbol{(i)}$ 
		Then by \thref{lemma : properties}(vi) and \thref{Remark 2}(i), $\phi$ induces a non-trivial exponential map on 
		$$
		\widetilde{B}:= \dfrac{\overline{k(x)}[Y,Z,T]}{(Y^b + Z^c +T^d +x^a)}.
		$$ 
		By the given condition $(b,c,d)\in F_3\setminus (R_3\cup T_3\cup S_3)$, hence by \thref{main}, $B_{(b,c,d)}$ is rigid. Therefore by \thref{Lemma 5}, $\widetilde{B}$ is rigid. 
		This leads to a contradiction.
		\smallskip
		
		\noindent
		$\boldsymbol{(ii)}$  
		Since $(a,b,c,d)\in F_4$, $p\nmid a$. 
		Let 	$A:= k[x, y^{p^m},z^{p^m},t^{p^m}]$ be a subring of $B$. 
		Then 
		$$
		A\cong \dfrac{k[X,Y,Z,T]}{(X^a + Y^{s_2} + Z^{s_3p^{r-m}} + T^{s_4p^{e-m}})}= B_{(a,s_2,s_3p^{r-m}, s_4p^{e-m})}.
		$$ 
		By \thref{gsubring criterion}, $\phi$ induces a non-trivial exponential map  $\phi_1$ on $A$ such that $x\in A\cap \B = A^{\phi_1}$. 
		Since $(a,b,c,d)\in F_4\setminus(R_4\cup T_4)$, $(s_2,s_3p^{r-m}, s_4p^{e-m})\in F_3\setminus (R_3\cup T_3\cup S_3)$.
		By (i) such a $\phi_1$ cannot exist on $A$, a contradiction. 
	\end{proof}
	
	Finally, we deduce a partial result on the rigidity of a certain Pham-Brieskorn surface when $k$ is not algebraically closed. 
	Note that if $k$ is algebraically closed then by Theorem B(ii), this surface is non-rigid.
	
	\begin{prop}
		Let $k$ be a field of characteristic $p\not=2$, not containing a square root of $-1$. 
		We consider
		$B := B_{(2,2,c)}= k[x,y,z]$ for an odd integer $c>1$, then $B$ is rigid.
	\end{prop}
	\begin{proof} 
		Suppose, if possible, that there exists a non-trivial exponential map $\phi$ on $B$ and let $f\in B^{\phi}\setminus k$. 
		Let $\phi_1 = \phi\otimes_k id$ be the extension of $\phi$ on $ B_1 := k[x,y,z]\otimes_k\overline{k}\cong \overline{k}[x,y,z]$ then $f\in B_1^{\phi_1}\setminus \overline{k}$. 
		Since $c$ is odd, as in the the proof of \thref{model lemma}, we get $\widehat{f} = \lambda_1(x+\mu y)^r\in B_1^{\widehat{\phi_1}}\subseteq B_1$ for some $\lambda_1\in k^{*}$, $\mu\in \overline{k}^{*}$ and $r\in \N$. Therefore  $x+\mu y\in B_1^{\widehat\phi_1}$ (cf. \thref{lemma : properties}(i)). 
		Now for every $\lambda$ in $\overline{k}^{*}$ 
		$$
		\frac{B_1}{(x+\mu y - \lambda)B_1} = \frac{\overline{k}[X,Y,Z]}{(X^2 + Y^2 + Z^c, X+\mu Y -\lambda )}\cong \frac{\overline{k}[Y,Z]}{(Z^c + P_{\lambda}(Y))},
		$$
		
		\noindent
		where $P_{\lambda}(Y) := (1 + \mu^2)Y^2 - 2\lambda\mu Y + \lambda^2\in \overline{k}[Y]$. 		
		Thus $P_{\lambda}(Y)\notin \overline{k}$. Hence $x + \mu y- \lambda$ is a prime element of $B_1$ for all $\lambda\in \overline{k}^{*}.$ 
		Therefore,  by \thref{betta-lemma} ,there exists  $\beta\in \overline{k}^*$ such that $\widehat{\phi_1}$ induces a non-trivial exponential map on the integral domain $\widetilde{B}_1:= B_1/(x+\mu y-\beta)B_1\cong {\overline{k}[Y,Z]}/{(Z^c + P_{\beta}(Y))}.$ 
		Since $\td_k\widetilde{B}_1$ = 1, so by \thref{lemma : properties}(v),  $\widetilde{B}_1=\overline{k}^{[1]}$. Suppose, if possible, $\mu^2 +1 \not= 0$. Then we would have
		$$
		\begin{array}{lll}	
			Z^c + P_{\beta}(Y)&
			=& 
			Z^c + (1 + \mu^2)Y^2 - 2\beta\mu Y + \beta^2\\
			&
			=&
			Z^c + (\mu^2 +1)(Y-\lambda_1)^2 + \lambda_2,
		\end{array}
		$$
		where $\lambda_1 := \mu\beta/(\mu^2 +1)$ and $\lambda_2 :={\beta^2(2\mu +1)}/{(\mu^2 +1)^2}$.
		So $Z^c + P_{\beta}(Y)$ would have no linear term in the new co-ordinates $Z$ and $Y_1:= Y-\lambda_1$.  
		Hence $\widetilde{B}_1\neq \overline{k}^{[1]}$, a contradiction. Therefore $\mu^2 +1 = 0$, say $\mu = i$, a square root of $-1$ in $\overline{k}$. 
		Note that $\widehat{\phi_1}$ is the extension of $\widehat{\phi}$ on $B_1$. Therefore $\widehat{f} = \lambda_1(x+iy)^r\in B^{\widehat{\phi}}\subseteq B$ for $\lambda_1\in k^{*}$ and $r>0$. 
		So $i\in k$, a contradiction. 
		Therefore $B_{(2,2,c)}$ is rigid for every odd integer $c>1$. 
	\end{proof}	
	
	\bigskip
	
	\noindent
	{\bf Acknowledgements.}
	The authors express their utmost gratitude to Professor Amartya K. Dutta and Parnashree Ghosh for going through the draft and giving valuable suggestions.

	{\small{
					
		\end{document}